\documentclass[11pt, reqno]{amsart}
\usepackage{microtype} 
\microtypesetup{protrusion=true, expansion=true}

\usepackage{amsmath, amsfonts, amsthm, amssymb, verbatim,
xcolor, multirow, booktabs, mathdots, bm, amscd, latexsym, mathrsfs, comment}

\usepackage[pdftex, 
            hidelinks,
            pdfencoding=auto,
            psdextra]{hyperref}
\hypersetup{
    colorlinks=true,
    linkcolor=blue, 
    citecolor=blue,
    urlcolor=blue,
    pdftitle={LOCAL CHARACTERIZATIONS OF THE LOCALLY SOLVABLE RADICAL},
    pdfauthor={Cao Minh Nam} 
}

\usepackage{fontenc}
\usepackage[T1]{fontenc}  
\usepackage{lmodern}

\usepackage[shortlabels]{enumitem}
\usepackage[capitalize]{cleveref}
\crefname{equation}{}{}

\usepackage{tikz-cd}

\numberwithin{equation}{section}
\newtheorem{theorem}{Theorem}[section]
\newtheorem{lemma}[theorem]{Lemma}

\newtheorem{proposition}[theorem]{Proposition}
\newtheorem{corollary}[theorem]{Corollary}
\newtheorem*{theorem*}{Theorem}

\newtheorem*{problem}{Problem}

\theoremstyle{definition}

\newtheorem{Remark}[theorem]{Remark}

\newcommand{\LS}{\textup{L}\mathfrak{S}}

\numberwithin{equation}{section}

\title[Locally solvable radicals]{Locally Solvable Radicals via Wilson Radical Sets and Subgroup Lattices}
\author{Cao Minh Nam}

\begin{document}

\begin{abstract} 
For a group $G$, let $S(G)$ be the set of elements $g \in G$ such that $\langle g, x \rangle$ is solvable for all $x \in G$. We study when $S(G)$ coincides with the locally solvable radical $R_{\mathrm{L}\mathfrak{S}}(G)$. Using Wilson's profinitely convergent word sequences, we show that, for every locally (solvable-by-finite) group $G$ and every Wilson sequence $\omega$, $$ R_{\mathrm{L}\mathfrak S}(G) = R_{\mathrm{L}\mathfrak R}(G) = S(G) = \mathcal W_\omega(G). $$ We also obtain four-conjugate and seven-commutator descriptions of radical membership, together with a two-conjugate result for torsion elements of order coprime to $6$. The same radical identity holds for locally linear groups, and hence for subgroups of $\mathrm{GL}_\infty(D)$ when $D$ is a locally finite-dimensional division ring. Independently, we prove that groups with nearly modular subgroup lattice satisfy $$R_{\mathrm{L}\mathfrak S}(G)=R_{\mathrm{L}\mathfrak R}(G)=S(G).$$
\end{abstract}

\maketitle

\section{Introduction}
Many results in group theory recover global structure from subgroups generated by only a few elements. 
Two classical examples concern nilpotence. For Noetherian groups, Baer's theorem \cite{Baer} identifies the Engel elements with the nilpotent radical. 
For finite groups, the Baer--Suzuki theorem \cite{Suzuki} states that, for a prime $p$, a $p$-element $g\in G$ lies in the largest normal $p$-subgroup $O_p(G)$ of $G$ if and only if  $\langle g,g^x\rangle$ is a $p$-group for every $x\in G$.

An analogous result for solvability is Thompson's theorem \cite{Thompson68}:~a finite group is solvable if and only if each of its two-generated subgroups is solvable. Flavell \cite{Flavell95} later gave a short direct proof  and proposed an elementwise analogue for the solvable radical \cite[Conjecture~B]{Flavell01}. Specifically, for a group $G$, let $R(G)$ denote its solvable radical and define
$$S(G) = \left\{ g\in G \mid \langle g,x\rangle \text{ is solvable for every }x\in G \right\}.$$
Flavell conjectured that $$R(G)=S(G)$$ for every finite group $G$.

Guralnick, Kunyavski\u{\i}, Plotkin, and Shalev proved the conjecture in \cite[Theorem~1.1]{Guralnick06}. 
The authors also established analogous results for finite-dimensional Lie algebras over fields of characteristic zero and for linear~groups. 

These  results may not be extended to arbitrary infinite groups. 
As observed in the proof of \cite[Proposition~5.3]{Guralnick06}, a Golod--Shafarevich construction provides an infinite three-generated residually finite group $G$ in which every two-generated subgroup is nilpotent, although $G$ itself is not solvable. 
Consequently, $S(G)=G,$ but $S(G)$ is not solvable. 
Thus, Thompson's criterion fails in this generality, and $S(G)$ need not itself be a solvable radical.

This observation suggests replacing the solvable radical by its local analogue. 
Whenever it exists, the \emph{locally solvable radical} $R_{\mathrm{L}\mathfrak{S}}(G)$ is defined to be the largest normal locally solvable subgroup of $G$. 
For finite groups and linear groups of finite degree, the locally solvable radical coincides with the solvable radical, since every locally solvable group in either class is solvable. 
In the linear case, this follows from Zassenhaus's theorem \cite[Satz~8]{Zassenhaus1938}. 
Moreover, \cite[Theorem~4.4] {Guralnick06} shows that $R_{\mathrm{L}\mathfrak{S}}(G)=S(G)$ for every PI-group $G$, which motivates the following question.

\medskip

\noindent\textbf{Question.}
For which classes of groups $G$ does the locally solvable radical exist and satisfy
  $R_{\mathrm{L}\mathfrak{S}}(G)=S(G)?$

\medskip

In the present paper, we shall investigate this question for several additional classes of groups using two independent methods. 
The first method uses Wilson radical sets. These are defined by profinitely convergent  sequences of words in two variables. 
For finite groups and, more generally, linear groups, eventual vanishing of such a sequence characterizes membership in the solvable radical \cite[Theorem]{Wilson2011}. 
Each such sequence $\omega$ determines a subset $\mathcal{W}_{\omega}(G)$ of an arbitrary group $G$ satisfying $S(G)\subseteq \mathcal{W}_{\omega}(G).$
For every locally (solvable-by-finite) group $G$, we prove that there exists the largest normal locally radical subgroup. 
We denote this subgroup by $R_{\mathrm{L}\mathfrak{R}}(G)$ and call it the \emph{$\mathrm{L}\mathfrak{R}$-radical} of $G$. 
Moreover, for every Wilson sequence $\omega$,
  $$R_{\mathrm{L}\mathfrak{S}}(G) = R_{\mathrm{L}\mathfrak{R}}(G) = S(G) = \mathcal{W}_{\omega}(G).$$
For the same class, we establish local analogues of the finite four-conjugate and seven-commutator criteria, together with a two-conjugate criterion for torsion elements of order coprime to $6$.

The same radical identity also holds for locally linear groups. 
In particular, if $D$ is a locally  finite-dimensional division ring, then every subgroup $G$ of the stable general linear group $\mathrm{GL}_{\infty}(D)$ satisfies
  $$R_{\mathrm{L}\mathfrak{S}}(G) = R_{\mathrm{L}\mathfrak{R}}(G) = S(G) = \mathcal{W}_{\omega}(G).$$

Our second method is lattice-theoretic and does not use Wilson radical sets. 
Combining results on subgroup lattices from \cite{Iwasawa41,Schmidt,Suzuki51} with the theory of nearly modular subgroups introduced by de Giovanni and Musella in \cite{Giovanni}, we prove that every group $G$ with nearly modular subgroup lattice satisfies
  $$R_{\mathrm{L}\mathfrak{S}}(G) = R_{\mathrm{L}\mathfrak{R}}(G) = S(G).$$

The remainder of the paper is organized as follows.  
Section $2$ develops Wilson radical sets and treats locally (solvable-by-finite) groups, establishing the local four-conjugate, seven-commutator, and two-conjugate criteria. 
Section $3$ considers locally linear groups and subgroups of $\mathrm{GL}_{\infty}(D)$. 
Section $4$ provides a lattice-theoretic argument for groups with nearly modular subgroup lattices.

\section{Locally (solvable-by-finite) groups}

In this section, we apply Wilson's characterization of the solvable radical in finite groups to locally (solvable-by-finite) groups.  
For each Wilson sequence $\omega$, we define a set $\mathcal W_\omega(G)$ through the eventual vanishing of the corresponding word values. 
Although this set is a priori dependent on $\omega$, we prove that, for every locally (solvable-by-finite) group $G$, 
$$ R_{\mathrm{L}\mathfrak S}(G) = R_{\mathrm{L}\mathfrak R}(G) = S(G) = \mathcal W_\omega(G).$$ 
Thus, the locally solvable radical is characterized by the condition defining $S(G)$. 
Analogous characterizations in terms of conjugates and commutators are established later in this section. 

\subsection{Wilson radical sets and local radicals} \hfill

\medskip
We begin by recalling the finite result on which the argument rests. 
Let $F(x,y)$ be the free group on two generators $x$ and $y$. 
A sequence $(w_n)$ in $F(x,y)$ is \emph{profinitely convergent} if, for every finite group $E$, there is a positive integer $N$ such~that $w_i(a,b)=w_j(a,b)$ for all $a,b\in E$ and all $i,j\geq N$. 

\begin{proposition}[{\cite[Eq.~(2) and Theorem]{Wilson2011}}]
\label{prop:Wilson-sequence}
There exists a sequence of~words $(w_n(x,y))_{n\geq 1}$ in the free group $F(x,y)$ with the following properties:
\begin{enumerate}[label=\textup{(\arabic*)}]
    \item The sequence $(w_n)_{n\geq 1}$ is profinitely convergent.
    \item For every integer $m\geq 1$, there exists an integer $N=N(m)$ such that
          $$w_n\in \left\langle [x,y]^{F(x,y)}\right\rangle^{(m)}$$ 
          for all $n\geq N$.
    \item For every finite group $G$ and every $g\in G$, one has $g$ belongs to the solvable radical $R(G)$ if and only if, for each $h\in G$, there exists an integer
          $N=N(g,h)$ such~that $w_n(g,h)=1$ for all $n\geq N$.
\end{enumerate}
\end{proposition}
  
Let $\Omega$ be the set of all sequences of words satisfying conditions \textup{(1)}--\textup{(3)} of Proposition~\ref{prop:Wilson-sequence}.  
For $\omega=(w_n(x,y))_{n\geq 1}\in\Omega$ and for an arbitrary group $G$, define the \emph{Wilson radical set associated with $\omega$}, henceforth
called simply the \emph{Wilson radical set}, by 
$$\mathcal W_\omega(G)=\left\{\,g\in G \mid \forall h\in G,\ \exists N(g,h):\ \forall n\geq N(g,h),\ w_n(g,h)=1 \,\right\}.$$

Wilson also extended Proposition~\ref{prop:Wilson-sequence}(3) to linear groups, drawing on Gruenberg's theorem \cite[Theorem~2]{Gruenberg66} that a finitely generated linear group is solvable if and only if all of its finite quotients are solvable.  We note this consequence for later use.

\begin{theorem}
\label{theo:Wilsonradsetlineargroup}
Let $\mathbb F$ be a field and let $G\leq \mathrm{GL}_n(\mathbb F)$.  Then, for every $\omega\in\Omega$,
$$\mathcal W_\omega(G)=R(G).$$
\end{theorem}

We now state two elementary facts about Wilson radical sets.

\begin{lemma}\label{lem:SsubseteqW}
Let $G$ be a group, and let $\omega\in\Omega.$ Then the following properties~hold:
\begin{enumerate}[label=\textup{(\arabic*)}]
    \item $S(G) \subseteq \mathcal{W}_{\omega}(G)$.
    \item The subgroup generated by $ \mathcal{W}_{\omega}(G)$ is characteristic in $G$.
\end{enumerate}
\end{lemma}
    
\begin{proof} (1) 
Fix $g\in S(G)$, let $h\in G$ be arbitrary, and put $H=\langle g,h\rangle$. 
Since $[g,h]\in\langle g^H\rangle$ and $\langle g^H\rangle\trianglelefteq H$, it follows that
  $$\bigl\langle [g,h]^H\bigr\rangle
  \leq
  \langle g^H\rangle.$$
Moreover, the solvability of $H$ implies that the closure $\bigl\langle [g,h]^H \bigr\rangle$ has derived length $m \geq 0$. 
Write $\omega=(w_n(x,y))_{n\geq 1}$. 
Since   $\bigl\langle [g,h]^H\bigr\rangle^{(m)}=1,$ Proposition~\ref{prop:Wilson-sequence}\textup{(2)} provides an integer $N=N(m)$ such that $w_n(g,h)=1$ for every $n\geq N$. 
This forces $ g \in \mathcal{W}_{\omega}(G) $. 
Consequently  $S(G) \subseteq \mathcal{W}_{\omega}(G).$

    \medskip

        (2) Fix an automorphism $ \varphi $ of $ G $. 
        It suffices to verify the inclusion $$\varphi\bigl(\mathcal{W}_{\omega}(G)\bigr) \subseteq \mathcal{W}_{\omega}(G).$$
 For any $a \in G$, the surjectivity of $\varphi$ allows us to write $a = \varphi(h)$ for some $h \in G$. 
 For $g\in\mathcal W_\omega(G)$, there exists an integer $N=N(g,h)$ such~that
  $$w_n(g,h)=1$$ whenever $n\geq N$.
By functoriality of word values, for every $n\geq N$, we~have $$ w_n(\varphi(g),a) = w_n(\varphi(g),\varphi(h)) = \varphi\bigl(w_n(g,h)\bigr) = 1. $$  
Therefore $\varphi(g)\in\mathcal W_\omega(G)$. The proof is now complete.
\end{proof}

\begin{Remark} \label{rm:24}
It follows from Theorem~\ref{theo:Wilsonradsetlineargroup} that, for every linear group $G$, the set $\mathcal W_\omega(G)$ is independent of the choice of $\omega\in\Omega$.
\end{Remark}

An \emph{ascending series} of a group $G$ is a family of subgroups $(G_\alpha)_{\alpha\leq\gamma}$, indexed by an ordinal $\gamma$, satisfying the following four conditions:
\begin{enumerate}
\item[(i)] $G_0=1$ and $G_\gamma=G$;
\item[(ii)] $G_\alpha\leq G_\beta$ whenever $\alpha\leq\beta\leq\gamma$;
\item[(iii)] $G_\alpha\trianglelefteq G_{\alpha+1}$ for every $\alpha<\gamma$;
\item[(iv)] $G_\lambda=\bigcup_{\alpha<\lambda}G_\alpha$ for every limit ordinal $\lambda\leq\gamma$.
\end{enumerate}
The subgroups $G_\alpha$ are
called the \emph{terms} of the series, and the quotient groups
$G_{\alpha+1}/G_\alpha,
\,
\alpha<\gamma,$
are called its \emph{factors}. A subgroup $A\leq G$ is \emph{ascendant} in $G$ if it occurs as a term of some ascending series of $G$.

Recall that a group is \emph{radical} if it has an ascending series with  locally nilpotent factors. For locally (solvable-by-finite) groups, local radicality and local solvability coincide by \cite[Theorem~2.6]{CNH2025}.

\begin{lemma} \label{lem:locradislocsolonlocsolbyfinite} Let $G$ be a locally \textup{(}solvable-by-finite\textup{)} group. Then $G$ is locally radical if and only if it is locally solvable. 
\end{lemma} 

We can now state the main identification for this class. 

\begin{theorem} \label{theo:locally-finite} 
Every locally \textup{(}solvable-by-finite\textup{)} group $G$ admits
the locally solvable radical.
Moreover, for each $\omega\in\Omega$,
  $$R_{\mathrm{L}\mathfrak{S}}(G)
  =
  R_{\mathrm{L}\mathfrak{R}}(G)
  =
  S(G)
  =
  \mathcal{W}_{\omega}(G).$$
\end{theorem}

\begin{proof}
By \cite[Lemma~4.6]{CNH2025}, the locally solvable radical $R_{\LS}(G)$ exists as the subgroup generated by all locally solvable normal subgroups of $G$. In view of Lemma~\ref{lem:locradislocsolonlocsolbyfinite}, it coincides with the $\textup{L}\mathfrak{R}$-radical $R_{\textup{L}\mathfrak{R}}(G)$. 

Next we claim that $R_{\LS}(G) \subseteq S(G)$. Indeed,
let $g \in R_{\LS}(G)$ and $x \in G$.  Observe that
$\langle g,x\rangle
\big/
\bigl(\langle g,x\rangle\cap R_{\LS}(G)\bigr)$
is cyclic. According to \cite[Lemma 2.5]{CNH2025}, the intersection $\langle g, x \rangle \cap R_{\LS}(G)$ is solvable.  This forces  the subgroup $ \langle g, x \rangle $ to be solvable, and hence the inclusion $ R_{\LS}(G) \subseteq S(G) $ holds.   Combining this inclusion with
Lemma~\ref{lem:SsubseteqW}, we obtain
$R_{\LS}(G)
\subseteq
S(G)
\subseteq
\mathcal{W}_{\omega}(G).$

Therefore, to complete the proof, it remains to verify $\mathcal{W}_{\omega}(G) \subseteq R_{\LS}(G).$
Put
$M=\bigl\langle \mathcal W_\omega(G)\bigr\rangle.$
We shall show that $M$ is locally solvable. Consider elements $g_1,\ldots,g_m\in\mathcal W_\omega(G)$ and put
$L=\langle g_1,\ldots,g_m\rangle.$ Since $G$ is locally (solvable-by-finite), the group $L$ is solvable-by-finite. Hence there exists a solvable normal subgroup $A\trianglelefteq L$ such that $Q=L/A$ is finite. Let $\pi\colon L\twoheadrightarrow Q$ be the canonical epimorphism. We claim that $\pi(g_i)\in R(Q)$ for every $i=1,\ldots,m$. Fix $i$ and let $q\in Q$. Choose $\ell\in L$ such that $\pi(\ell)=q$. Since $g_i\in\mathcal W_\omega(G)$, the definition of the Wilson radical set provides an integer $N_i=N(g_i,\ell)$ such that $w_n(g_i,\ell)=1$ for every $n\geq N_i$. Therefore, $$w_n\bigl(\pi(g_i),q\bigr) = w_n\bigl(\pi(g_i),\pi(\ell)\bigr) = \pi\bigl(w_n(g_i,\ell)\bigr) = 1$$ for every $n\geq N_i$. 
Thus $\pi(g_i)\in\mathcal W_\omega(Q).$ 
On the other hand, since $Q$ is finite, it follows from Proposition~\ref{prop:Wilson-sequence}\textup{(3)} that $\pi(g_i)\in R(Q).$

Now $Q$ is generated by $\pi(g_1),\ldots,\pi(g_m)$, all of which belong to $R(Q)$. Hence $Q=R(Q),$ so $Q$ is solvable. Since $A$ is solvable, it follows that $L$ is solvable. Therefore $M$ is locally solvable. Moreover, 
$M$ is characteristic in $G$ by Lemma~\ref{lem:SsubseteqW}. Thus $M$ is a normal locally solvable subgroup of $G$, and so $M\leq R_{\LS}(G).$ In particular, $\mathcal W_\omega(G)\subseteq R_{\LS}(G),$ as required.
\end{proof}

The preceding result has two immediate consequences. 

\begin{corollary}\label{cor:locally_finite_2gen}
Let $G$ be a locally \textup{(\emph{solvable-by-finite})} group. Then:
\begin{enumerate}[label = \textup{(\arabic*)}]
    \item $G$ is locally solvable if and only if every two-generated subgroup of $G$ is solvable.
    \item If in each conjugacy class of $G$ every two elements generate a solvable group, then $G$ is locally solvable.
\end{enumerate}    
\end{corollary}

\begin{proof}
Part (1) follows immediately from Theorem~\ref{theo:locally-finite}. For part \textup{(2)},   fix an element $g \in G$, and let $K \le \langle g^G \rangle$ be an arbitrary finitely generated subgroup. Then $K$ is solvable-by-finite. Note that in every subgroup of $G$, any two conjugates generate a solvable subgroup.  By \cite[Corollary~2.10]{CNH2025}, the subgroup $K$ is solvable. This implies that the normal closure $\langle g^G \rangle$ is locally solvable, and hence $g \in R_{\textup{L}\mathfrak{S}}(G)$. Therefore $G$ coincides with its locally solvable radical, and so $G$ is locally solvable.
\end{proof}

\begin{corollary}\label{cor:Wilsonradseteqradset}
Let $G$ be a locally \textup{(\emph{solvable-by-finite})} group. Then the set $\mathcal{W}_{\omega}(G)$ is independent of the choice of
$\omega\in\Omega$; equivalently,
  $$\mathcal{W}_{\omega}(G) = \mathcal{W}_{\omega'}(G)$$
for all $\omega,\omega'\in\Omega$.
\end{corollary} 

To study normal closures of ascendant locally solvable subgroups, we introduce the following terminology. A group $G$ is called \emph{$R_{\mathrm{L}\mathfrak S}$-hereditary} if $R_{\mathrm{L}\mathfrak S}(H)$ exists for every subgroup $H\leq G$; equivalently, every subgroup of $G$ has the largest normal locally solvable
subgroup.

\begin{lemma}\label{lem:ascenlocsolsubgroupislocsol}
Let $G$ be an $R_{\textup{L}\mathfrak{S}}$-hereditary group. If $A$ is a locally solvable ascendant subgroup of $G$, then the normal closure $\langle A^G \rangle$ is locally solvable.
\end{lemma}

\begin{proof} Since $A$ is ascendant in $G$, there exist an ordinal $\gamma$ and a family of subgroups $(A_\alpha)_{\alpha\leq\gamma}$ satisfying 
$$ \left\{ 
  \begin{array}{ll} A_0=A,\,\, A_\gamma=G, \\[1mm] A_{\alpha_1}\leq A_{\alpha_2} & \text{whenever } \alpha_1\leq\alpha_2\leq\gamma, \\[1mm] A_\alpha\trianglelefteq A_{\alpha+1} & \text{for every } \alpha<\gamma, \\[1mm] A_\lambda=\displaystyle\bigcup_{\alpha<\lambda}A_\alpha & \text{for every limit ordinal } \lambda\leq\gamma. 
  \end{array} \right. $$ 
For each $\alpha\leq\gamma$, put $ B_\alpha=\langle A^{A_\alpha}\rangle,$ the normal closure of $A$ in $A_\alpha$. We prove by transfinite induction that $B_\alpha$ is locally solvable for every $\alpha\leq\gamma$. 

Since $B_0=A$, the assertion holds for $\alpha=0$. 

Let $\lambda\leq\gamma$ be a limit ordinal, and suppose that $B_\alpha$ is locally solvable for every $\alpha<\lambda$. The monotonicity of $(A_\alpha)_{\alpha\leq\gamma}$ implies that $(B_\alpha)_{\alpha<\lambda}$ is an ascending chain. Moreover, $$ B_\lambda = \bigl\langle A^{A_\lambda}\bigr\rangle = \bigcup_{\alpha<\lambda} \bigl\langle A^{A_\alpha} \bigr\rangle = \bigcup_{\alpha<\lambda}B_\alpha. $$ Indeed, every element of $A_\lambda$ lies in some $A_\alpha$ with $\alpha<\lambda$. Hence $B_\lambda$ is a directed union of locally solvable groups, and is therefore locally solvable. 

Now let $\alpha=\beta+1$, and assume that $B_\beta$ is locally solvable. Moreover, we have $B_{\beta+1}=\bigl\langle B_\beta^{A_{\beta+1}}\bigr\rangle$, and therefore $B_\beta^x\leq B_{\beta+1}$ for every $x\in A_{\beta+1}$. Since \(B_\beta^x\trianglelefteq A_\beta\) and $B_\beta^x\leq B_{\beta+1}\leq A_\beta,$ it follows that $B_\beta^x\trianglelefteq B_{\beta+1}.$
The subgroup $B_\beta^x$ is locally solvable. Since $G$ is $R_{\mathrm{L}\mathfrak{S}}$-hereditary, the locally solvable radical
$R_{\mathrm{L}\mathfrak{S}}(B_{\beta+1})$ exists, and the preceding result
gives
$B_\beta^x \leq R_{\mathrm{L}\mathfrak{S}}(B_{\beta+1}).$ As these conjugates generate
$B_{\beta+1}$, it follows that
$B_{\beta+1} \leq R_{\mathrm{L}\mathfrak{S}}(B_{\beta+1}).$
Thus $B_{\beta+1}$ is locally solvable. This completes the transfinite
induction, and hence the normal closure
$\langle A^G\rangle=B_\gamma$
is locally solvable. 
\end{proof}

The preceding lemma gives the following corollary.

\begin{corollary}\label{cor:ascendantlocsolsubgroupwithnorclosislocsol}
Let $G$ be a locally \textup{(\emph{solvable-by-finite})} group. If $A$ is an ascendant locally radical subgroup of $G$, then $\langle A^G \rangle$ is locally solvable.
\end{corollary}

\subsection{Conjugate and commutator characterizations} \hfill

\medskip
Recall that a group is \emph{completely reducible}, or a \emph{\textup{CR}-group}, if it is a restricted direct product of possibly infinite simple groups (see \cite[pp.~85 and~88]{Robinson:CourseTheoryGroups}). Every group $G$ admits a unique maximal normal centerless \textup{CR}-subgroup, called the \emph{centerless \textup{CR}-radical} of $G$ (see \cite[3.3.17]{Robinson:CourseTheoryGroups}).

Although the seven-commutator characterization follows implicitly from the reduction carried out in Section~2 of \cite{Gordeev08}, we give the details
because the precise formulation in terms of seven commutators is not stated explicitly there. Before stating the result, we fix our notation for  conjugates and commutators. For $a,b\in G$, we write
  $b^a=aba^{-1}$ and $[a,b]=aba^{-1}b^{-1}.$
\begin{theorem}\label{prop:7commutatorcriterion}
Let $G$ be a finite group, and let $g \in G$. If the subgroup $$\bigl\langle [g,x_1], \dots, [g,x_7] \bigr\rangle$$ is solvable for every $x_1, \dots, x_7 \in G$, then $g \in R(G)$.
\end{theorem}
\begin{proof}
Without loss of generality, assume that $ G $ is  semisimple. Suppose, toward a contradiction, that the assertion fails. Let $ G $ be a minimal counterexample. Thus, there exists a non-trivial element $ g\in G $ such that the~subgroup $$\langle[g,x_1],\dots,[g,x_7]\rangle$$ is solvable for all $ x_1,\dots,x_7\in G $.

By arguments analogous to those of the reduction theorem in \cite[Section~2]{Gordeev08}, we conclude both that the centerless CR-radical $V$ of $G$ is a direct product of isomorphic non-abelian simple groups, and that $g$ normalizes each direct factor. Moreover, in view of \cite[3.3.18(i)]{Robinson:CourseTheoryGroups}, the centralizer $C_G(V)$ is trivial. Thus, we can find a simple direct factor $S_0$ and an element $x \in S_0$ 
such that the commutator $c = [g,x]$ is non-trivial. According to \cite[Theorem 1.11]{Gordeev08}, there exist elements $y_1, y_2, y_3 \in S_0$ such that the subgroup 
$$\bigl\langle [c,y_1],[c,y_2],[c,y_3]\bigr\rangle$$ is non-solvable.
The formula $y[g,x]y^{-1}=[g,y]^{-1}[g,yx]$ now shows that 
$$\bigl[[g,x],y_i\bigr]\in\bigl\langle[g,x],[g,y_i],[g,y_ix]\bigr\rangle \text{ for } i=1,2,3.$$
From this, we have the inequality
$$\bigl\langle[c,y_1],[c,y_2],[c,y_3]\bigr\rangle
\leq
\bigl\langle[g,x],[g,y_i],[g,y_ix] \bigm| i=1,2,3\bigr\rangle.$$
The solvability of the right-hand side stands in contradiction to the nonsolvability of the left-hand side. This completes the proof.
\end{proof}

\begin{Remark}
The bound seven is not expected to be optimal. In \cite{Gordeev08}, Gordeev, Grunewald, Kunyavski\u{\i}, and Plotkin
asked whether the bound seven can be reduced to three. More precisely, the authors asked whether, for every finite group $G$ and every $g\in G$, $g\in R(G)$ if and only if $\langle [g,x_1],[g,x_2],[g,x_3]\rangle$  is solvable for all $x_1,x_2,x_3\in G.$
\end{Remark}

Together, Theorem~\ref{prop:7commutatorcriterion} and \cite[Theorem~1.1]{Gordeev09} give the following equivalent conditions for
membership in $R_{\mathrm{L}\mathfrak S}(G)$.

\begin{theorem}\label{theo:chacoflocsolrad}
Let $G$ be a locally \textup{(\emph{solvable-by-finite})} group, and let $g \in G$. Then, the following statements are equivalent:
\begin{enumerate}[label= \textup{(\arabic*)}]
    \item $g \in R_{\textup{L}\mathfrak{S}}(G).$     \item The normal closure $\langle g^G \rangle$ is locally solvable. 
    \item The normal closure $\langle g^{\langle x \rangle} \rangle$ of $g$ in $\langle g, x \rangle$ is solvable for all $x \in G$.   
    \item Every four conjugates of $g$ generate a solvable subgroup.
    \item The subgroup $\bigl\langle [g,x_1], \ldots, [g, x_7] \bigr\rangle$ is solvable for all $x_1, \ldots, x_7 \in G$.
\end{enumerate}
\end{theorem}

\begin{proof} By Theorem~\ref{theo:locally-finite}, $S(G) = R_{\textup{L}\mathfrak{S}}(G)$, which immediately implies the equivalence of \textup{(1)}, \textup{(2)}, and \textup{(3)}. Only $\textup{(4)}\implies \textup{(1)}$ and
$\textup{(5)}\implies \textup{(1)}$ require~proof.

For $\textup{(4)} \implies \textup{(1)}$, assume that any four conjugates of $g$ generate a solvable subgroup. Given any $x \in G$, let $H = \langle g, x \rangle$. Since $H$ is solvable-by-finite, it follows from \cite[Theorem~2.8]{CNH2025} that $H/R(H)$ is finite. The image $\overline{g}$ of $g$ in $H/R(H)$ inherits the four-conjugate condition. Thus, \cite[Theorem~1.1]{Gordeev09} implies $\overline{g} \in R(H/R(H))$, so $H = \langle g, x \rangle$ is solvable. Therefore $g \in R_{\textup{L}\mathfrak{S}}(G)$.

For $\textup{(5)} \implies \textup{(1)}$, suppose every subgroup generated by seven commutators of $g$ is solvable.
Suppose, toward a contradiction, that $g \notin R_{\textup{L}\mathfrak{S}}(G)$. By Theorem~\ref{theo:locally-finite}, there exists $x_0 \in G$ such that $H_0 = \langle g, x_0 \rangle$ is non-solvable. Since $H_0$ is solvable-by-finite, \cite[Theorem~2.8]{CNH2025} implies that $H_0/R(H_0)$ is finite, and the image $\overline{g}$ remains non-trivial in this quotient. It then follows from Theorem~\ref{prop:7commutatorcriterion} that there exist elements $\overline{b}_1, \ldots, \overline{b}_7 \in H_0/R(H_0)$ for which the subgroup $\bigl\langle [\overline{g}, \overline{b}_1], \ldots, [\overline{g}, \overline{b}_7] \bigr\rangle$ is non-solvable, a contradiction.
\end{proof}

Under suitable restrictions on the order of $g$, the four-conjugate condition in Theorem~\ref{theo:chacoflocsolrad}\textup{(4)} can be replaced by a two-conjugate condition. 

\begin{theorem}[{\cite[Theorem 1.4]{Gordeev2010}}]\label{prop:finite-prime}
Let $G$ be a finite group, and let $g\in G$ have prime order $p>3$.
If $\langle g,g^x\rangle$ is solvable for every $x\in G$, then $g\in R(G)$.
\end{theorem}

For locally (solvable-by-finite) groups, the finite result in Theorem~\ref{prop:finite-prime} leads to an exact two-conjugate characterization of membership in the locally solvable radical, with the exceptional primes $2$ and $3$ accounted for explicitly.

\begin{theorem}\label{thm:two-conjugate-23prime}
Let $G$ be a locally \textup{(\emph{solvable}-\emph{by}-\emph{finite})} group, and let $g\in G$. For each $x\in G$, set $H_x:=\langle g,x\rangle$. Then $g\in R_{\textup{L}\mathfrak{S}}(G)$ if and only if for every $x\in G$, the following two conditions hold: 
\begin{enumerate}[label = \textup{(\roman*)}]
    \item The subgroup $\langle g,g^x\rangle$ is solvable.
    \item $\gcd\bigl(|gR(H_x)|,6\bigr)=1.$ \end{enumerate}
\end{theorem}

The proof of Theorem~\ref{thm:two-conjugate-23prime} reduces to two finite group results, which we state next.

\begin{lemma}\label{lem212}
Let $G$ be a finite group, and let $g \in G$. If the subgroup $\langle g, g^x \rangle$ is solvable for every $x \in G$, then $gR(G)$ is a $\{2,3\}$-element. 
\end{lemma}

\begin{proof}
Let $\overline{G}=G/R(G)$ and $\overline{g}=gR(G)$. Then $\langle\overline{g},\overline{g}^{\,\overline{x}}\rangle$ is solvable for every $\overline{x}\in\overline{G}$. It suffices to  show that no prime $p>3$ divides $|\overline{g}|$.

Suppose, toward a contradiction, that a prime $p>3$ divides $|\overline{g}|$. Put $n=|\overline{g}|$ and $\overline{h}=\overline{g}^{n/p}$. Since $\overline{h}$ is a power of $\overline{g}$, the subgroup $\langle\overline{h},\overline{h}^{\,\overline{x}}\rangle$ is solvable for every $\overline{x}\in\overline{G}$. By Theorem~\ref{prop:finite-prime}, we have $\overline{h}\in R(\overline{G})$. But the group $\overline{G}$ has trivial solvable radical, so that $\overline{h}=1$, which contradicts  $|\overline{h}|=p$. Thus no prime greater than $3$ divides $|\overline{g}|$, and therefore $\overline{g}$ is a $\{2,3\}$-element.
\end{proof}

The coprimality condition $\gcd(|g|,6)=1$ eliminates the possible
$\{2,3\}$-part of $gR(G)$ and gives the following characterization for finite groups.

\begin{lemma}\label{lem:twoconjugateschar}
Let $G$ be a finite group and let $g\in G$ with $\gcd(|g|,6)=1$. If the subgroup $\langle g,g^x\rangle$ is solvable for every $x\in G$, then $g \in R(G)$.
\end{lemma}
    
\begin{proof}
Let $\overline{g}=gR(G)$ be the image of $g$ in $G/R(G)$. Since
$|\overline{g}|$ divides $|g|$, the assumption $\gcd(|g|,6)=1$ gives
$\gcd(|\overline{g}|,6)=1.$
On the other hand, Lemma~\ref{lem212}, applied to the hypothesis that
$\langle g,g^x\rangle$ is solvable for every $x\in G$, shows that
$|\overline{g}|$ is a $\{2,3\}$-number. Therefore $|\overline{g}|=1$, and so $g\in R(G)$.
\end{proof}

\begin{proof}[Proof of Theorem~\ref{thm:two-conjugate-23prime}]
Assume first that $g \in R_{\textup{L}\mathfrak{S}}(G)$.~Fix  $x \in G$.~Then the subgroup $\langle g, g^x \rangle$ 
is solvable. Therefore \textup{(i)} holds. Moreover, since $G$ is locally (solvable-by-finite), the subgroup $H_x$ is solvable-by-finite. According to \cite[Lemma 2.5]{CNH2025}, we have $H_x \cap R_{\textup{L}\mathfrak{S}}(G) \leq R(H_x).$ Observe that $g$ belongs to this intersection. Thus $g\in R(H_x)$, so $gR(H_x)=R(H_x)$ and $|gR(H_x)|=1.$ Hence $\gcd\bigl(|gR(H_x)|,6\bigr)=1,$ which proves \textup{(ii)}.

Next assume that conditions (i) and (ii) hold. Fix an element $x \in G$. Since $H_x$ is solvable-by-finite,  $Q = H_x/R(H_x)$ is finite. Let $a =~gR(H_x) \in~Q$.  Condition (ii) now implies that $\gcd(|a|, 6) = 1$.

For any $y \in H_x$, condition (i) implies that $\langle g, g^y \rangle$ is solvable, and hence its image $\langle a, a^{yR(H_x)} \rangle$ is solvable.~Therefore,  the subgroup $\langle a, a^q \rangle$ is solvable for any $q \in Q$.~It follows from Lemma~\ref{lem:twoconjugateschar} that $a \in R(Q)$.~By \cite[Theorem~2.8]{CNH2025}, we have $R(Q)=1$.~Hence $a=1$, and $g\in R(H_x)$.~Consequently, $H_x/R(H_x)$ is cyclic, and thus $H_x$ is solvable.~By Theorem~\ref{theo:locally-finite}, $g \in R_{\textup{L}\mathfrak{S}}(G)$, as~required.
\end{proof}

If $g$ is torsion and $\gcd(|g|,6)=1$, then condition~\textup{(ii)} of
Theorem~\ref{thm:two-conjugate-23prime} holds automatically.  We therefore
obtain the following two-conjugate consequence.

\begin{corollary}\label{cor:two_conj_torsion}
Let $G$ be a locally \textup{(}solvable-by-finite\textup{)} group, and let $g \in G$. Suppose that $g$ is torsion with $\gcd \bigl(|g|,6 \bigr)=1$. If the subgroup $\langle g, g^x \rangle$ is solvable for every $x \in G$, then $g \in R_{\textup{L}\mathfrak{S}}(G)$.
\end{corollary}

\section{Locally linear groups and stable general linear groups}

This section is devoted to locally linear groups and stable general linear groups over division rings. Recall that a group is \emph{locally linear} if each of its finitely generated subgroups admits a faithful finite-dimensional linear representation over some field. The main ingredient is Zassenhaus's theorem \cite[Satz~8]{Zassenhaus1938}, which states that every locally solvable subgroup of $\mathrm{GL}_n(\mathbb F)$ is solvable, where $\mathbb F$ is a field. This allows the preceding radical results to be applied within finitely generated linear subgroups. Throughout this section, $D$ denotes a division ring. We identify $\mathrm{GL}_n(D)$ with a subgroup of $\mathrm{GL}_{n+1}(D)$ via the embedding $A\longmapsto \operatorname{diag}(A,1).$ The \emph{stable general linear group} over $D$ is then defined by
  $$\mathrm{GL}_{\infty}(D) = \varinjlim_{n}\mathrm{GL}_n(D) = \bigcup_{n\geq 1}\mathrm{GL}_n(D).$$

Recall that a group $G$ is \emph{hyper-\textup{(}locally solvable\textup{)}} if it admits an ascending normal series with locally solvable factors. We shall use the following consequence of Zassenhaus's theorem.

\begin{lemma}
\label{lem:radlinearissol}
Every hyper-\textup{(}locally solvable\textup{)} subgroup of
$\mathrm{GL}_n(\mathbb F)$ is solvable.
\end{lemma}

\begin{proof}
Let $G\leq \mathrm{GL}_n(\mathbb F)$, and let
$(G_\alpha)_{\alpha\leq\gamma}$ be an ascending series of $G$ with locally solvable factors.  We proceed by transfinite induction to show that every $G_\alpha$ is solvable. The assertion is clear in case where $\alpha =0$.  

Suppose that $\lambda\leq\gamma$ is a limit ordinal and that $G_\alpha$ is solvable for every $\alpha<\lambda$.
Since $G_\lambda=\cup_{\alpha<\lambda}G_\alpha,$ every finitely generated subgroup of $G_\lambda$ is contained in some $G_\alpha$ with $\alpha<\lambda$.  Thus $G_\lambda$ is locally solvable, and hence solvable by Zassenhaus's theorem.

Now let $\alpha=\beta+1$ and suppose that $G_\beta$ is solvable.  For every
finitely generated subgroup $H\leq G_{\beta+1}$, the subgroup
$H\cap G_\beta$ is solvable, while
$$H/(H\cap G_\beta) \cong HG_\beta/G_\beta $$
is a finitely generated subgroup of the locally solvable group
$G_{\beta+1}/G_\beta$, and is therefore solvable.  Hence $H$ is solvable.
Thus $G_{\beta+1}$ is locally solvable, and hence solvable by Zassenhaus's theorem. Consequently $G$ is solvable.
\end{proof}

Recall that a group is radical if and only if it is
hyper-(locally nilpotent). The preceding lemma therefore collapses local radicality to local solvability inside locally linear groups.

\begin{lemma} \label{lem:locradislocsol}
    A locally linear group is locally radical if and only if it is locally solvable.
\end{lemma}

As noted at the beginning of this section, Theorem~\ref{theo:Wilsonradsetlineargroup} applies to every finitely
generated subgroup of a locally linear group, while the preceding lemma shows that local radicality and local solvability coincide.

\begin{theorem} \label{theo:locradislocsol} Let $G$ be a locally linear group. Then $G$ is $R_{\mathrm{L}\mathfrak S}$-hereditary and, for every $\omega\in\Omega$, $$ R_{\mathrm{L}\mathfrak S}(G) = R_{\mathrm{L}\mathfrak R}(G) = S(G) = \mathcal W_\omega(G). $$ 
\end{theorem}

\begin{proof}
Let $\mathcal{N}$ denote the collection of locally solvable normal subgroups of $G$, and let $R$ be the subgroup of $G$ generated by all members of $\mathcal{N}$.
Clearly, the subgroup $R$ is normal in $G$. 
Let $X$ be a finite subset of $R$. Then there exist $N_1,\ldots,N_m\in\mathcal{N}$ such that $X$ is contained in a subgroup $N$ generated by finitely many elements of $\cup_{i=1}^{m}N_i.$

By the local linearity of $G$, the subgroup $N$ embeds into $\mathrm{GL}_n(\mathbb{F})$ for some field $\mathbb{F}$ and integer $n \geq 1$.    
For each $i$, the intersection $N\cap N_i$ is a locally solvable subgroup of a linear group over $\mathbb{F}$, and is therefore solvable by Zassenhaus's theorem. Note that $N=\bigl\langle N \cap N_i \mid 1 \leq i\leq m \bigr\rangle.$  Therefore $N$ is solvable, and hence $R$ is locally solvable. Consequently  $G$ admits the locally solvable radical with $R_{\textup{L}\mathfrak{S}}(G)=R.$       
Moreover, the class of locally linear groups is subgroup-closed, so $G$ is $R_{\mathrm{L}\mathfrak{S}}$-hereditary. On the other hand, in view of Lemma~\ref{lem:locradislocsol}, 
$R_{\mathrm{L}\mathfrak{S}}(G)=R_{\mathrm{L}\mathfrak{R}}(G).$ 

Next let $g \in R_{\LS}(G)$, so that its normal closure $\langle g^G \rangle$ is locally solvable. For any $x \in G$, the subgroup $\langle g, x \rangle$ is  linear. The normal closure of $g$ in $\langle g,x\rangle$ is $\langle g^{\langle x\rangle}\rangle$. Since
$\langle g^{\langle x\rangle}\rangle \leq \langle g^G\rangle,$ it is locally solvable and, being linear, is solvable by Zassenhaus's theorem.~Since the quotient $\langle g, x \rangle/\langle g^{\langle x \rangle} \rangle$ is cyclic, the solvability of $\langle g, x \rangle$ follows. This forces $g \in S(G)$, and hence $R_{\LS}(G) \subseteq S(G)$. Together with Lemma~\ref{lem:SsubseteqW}, this gives
$$R_{\LS}(G)
\subseteq
S(G)
\subseteq
\mathcal{W}_{\omega}(G).$$
It remains only to verify the inclusion $\mathcal{W}_{\omega}(G) \subseteq R_{\LS}(G)$.~In view of Lemma~\ref{lem:SsubseteqW}, this task reduces to proving that the subgroup $\langle \mathcal{W}_{\omega}(G) \rangle$ is locally solvable. Indeed, choose elements $g_1, \dots, g_m \in \mathcal{W}_\omega(G)$ and consider $$L = \left\langle g_1, \dots, g_m \right\rangle.$$ Local linearity of $G$ forces $L$ to be linear. For each $i$, we have $g_i\in\mathcal{W}_\omega(L)$, and Theorem~\ref{theo:Wilsonradsetlineargroup} yields $\mathcal{W}_\omega(L)=R(L)$. Consequently, $g_i\in R(L)$ for every $i$, and hence $L=R(L)$. Thus $L$ is solvable. Therefore $\langle\mathcal{W}_{\omega}(G)\rangle$ is locally solvable, and this completes the proof.
\end{proof}

We now apply this structural result to stable general linear groups. This requires examining the dimensionality properties of the underlying division~ring. Following Wehrfritz \cite[p.~325]{Wehrfritz94}, a division ring $D$
is called \emph{locally finite-dimensional} if every finite subset
$X\subseteq D$ lies in a division subring $S$ of $D$ such that
$[S:Z(S)]<\infty.$
Equivalently, for every finite subset $X\subseteq D$, the division
subring of $D$ generated by $X$ is finite-dimensional over its own
center (see \cite[Theorem~1, p.~228]{Jacobson1964}).

\begin{corollary} \label{cor:linearoverweaklocfinlocsolrad}
Let $D$ be a locally finite-dimensional division ring. Then the stable general linear group $\mathrm{GL}_\infty(D)$ is locally linear. In particular, every subgroup $G \leq \mathrm{GL}_\infty(D)$ is $R_{\textup{L}\mathfrak{S}}$-hereditary, and 
$$ R_{\textup{L}\mathfrak{S}}(G) = R_{\textup{L}\mathfrak{R}}(G) = S(G) = \mathcal{W}_{\omega}(G). $$
\end{corollary}
    
\begin{proof}
Let $H = \langle h_1, \dots, h_r \rangle \leq \mathrm{GL}_\infty(D)$. Since each $h_i$ lies in $\mathrm{GL}_{n_i}(D)$ for some integer $n_i$, it follows that $H \leq \mathrm{GL}_n(D)$ with $n = \max\{n_1, \dots, n_r\}$.

Let $L \subseteq D$ be the division subring generated by the entries of $h_i^{\pm 1}$ for $1 \le i \le r$. Then $H \leq \mathrm{GL}_n(L)$.
The local finite-dimensionality of $D$ implies that $L$ is
finite-dimensional over its center $\mathbb{F}=Z(L)$. Put
$[L:\mathbb{F}]=m<\infty$. It follows that $L^n$ is an
$\mathbb{F}$-vector space of dimension $nm$. Left multiplication
defines a faithful $\mathbb{F}$-linear action of
$\mathrm{GL}_n(L)$ on $L^n$: $\mathbb{F}$-linearity follows from
$\mathbb{F}=Z(L)$, while faithfulness follows by evaluating the action
on the standard basis of $L^n$. Hence, after choosing an
$\mathbb{F}$-basis of $L^n$, we obtain
$$H\leq \mathrm{GL}_n(L) \hookrightarrow \mathrm{GL}_{nm}(\mathbb{F}).$$
As a result, $H$ is linear over the field $\mathbb{F}$, and thereby the stable general linear group $\mathrm{GL}_\infty(D)$ is locally linear.
Moreover, every subgroup of a locally linear group is locally linear. By Theorem~\ref{theo:locradislocsol}, every subgroup $G\leq \mathrm{GL}_{\infty}(D)$ is
$R_{\mathrm{L}\mathfrak{S}}$-hereditary and satisfies
  $$R_{\mathrm{L}\mathfrak{S}}(G)  = R_{\mathrm{L}\mathfrak{R}}(G) = S(G) = \mathcal{W}_{\omega}(G).$$
This completes the proof.
\end{proof}

\begin{corollary} \label{cor:gl_n_locfin}
Let $D$ be a locally finite-dimensional division ring. Then every subgroup $G$ of $\mathrm{GL}_n(D)$ is $R_{\textup{L}\mathfrak{S}}$-hereditary, and satisfies 
$$ R_{\textup{L}\mathfrak{S}}(G) = R_{\textup{L}\mathfrak{R}}(G) = S(G) = \mathcal{W}_{\omega}(G). $$
\end{corollary}

The preceding corollaries assume that $D$ is locally finite-dimensional. For $\mathrm{GL}_{\infty}(D)$ itself, however, this assumption is unnecessary.

\begin{theorem} \label{theo:locsolradofstableistri} For a division ring $D$, the stable general linear group $\mathrm{GL}_\infty(D)$ has no non-trivial subnormal locally solvable subgroups. In~particular, $$R_{\mathrm{L}\mathfrak S}\bigl(\mathrm{GL}_\infty(D)\bigr)=1.$$ 
\end{theorem}

\begin{proof}
Let $N$ be a locally solvable subnormal subgroup of $\mathrm{GL}_\infty(D)$, and let $a \in N$. Then there exists an integer $n$ 
such that $a \in \mathrm{GL}_n(D)$. Choosing an integer $m > n$ with $m \geq 3$, 
the element $a$, when viewed inside $\mathrm{GL}_m(D)$, assumes the block-diagonal form
$$\begin{pmatrix} 
  a & 0 \\ 0 & I_{m-n} 
\end{pmatrix}. $$
Note that the intersection
$N_m = N \cap \mathrm{GL}_m(D)$
is a subnormal subgroup of $\mathrm{GL}_m(D)$. If $N_m$ is noncentral, then induction on its subnormal defect, together
with \cite[Theorem~II.10.1]{Suprunenko}, shows that
$$\mathrm{SL}_m(D)\leq N_m\leq N.$$ 
This would imply that $\mathrm{SL}_m(D)$ is locally  solvable, and hence so is its quotient $\mathrm{PSL}_m(D)$. This contradicts the nonabelian simplicity of $\mathrm{PSL}_m(D)$ established in \cite[Theorem~II.10.4]{Suprunenko}.  Consequently $N_m$ is central. Moreover, since $a \in N_m$, it follows that  $a = \lambda I_m$ for some $\lambda \in Z(D)^*$. Thus $\lambda = 1$, and hence $a = I_m$. Therefore $N = \{1\}$, and $R_{\textup{L}\mathfrak{S}}(\mathrm{GL}_\infty(D)) = \{1\}$.
\end{proof}

\begin{corollary} \label{cor:ascendant_trivial} 
    Let $D$ be a locally finite-dimensional division ring. Then $\mathrm{GL}_\infty(D)$ has no non-trivial ascendant locally radical subgroups. 
\end{corollary}

\begin{proof} Put $G=\mathrm{GL}_\infty(D). $ By Corollary~\ref{cor:linearoverweaklocfinlocsolrad}, the group $G$ is locally linear and $R_{\mathrm{L}\mathfrak S}$-hereditary. Let $A$ be an ascendant locally radical subgroup of $G$. Note that subgroups of locally linear groups are locally linear. Lemma~\ref{lem:locradislocsol} now implies that $A$ is locally solvable. By Lemma~\ref{lem:ascenlocsolsubgroupislocsol}, the normal closure $\langle A^G \rangle$ is locally solvable.  By Theorem~\ref{theo:locsolradofstableistri},
$\langle A^G\rangle=1.$
Hence $A=1$, and therefore $\mathrm{GL}_\infty(D)$ has no non-trivial ascendant locally radical subgroups. 
\end{proof}

\section{Groups with Nearly Modular Subgroup Lattice}
This section is devoted to proving that
  $$R_{\mathrm{L}\mathfrak S}(G)
  =
  R_{\mathrm{L}\mathfrak R}(G)
  =
  S(G)$$
for groups with nearly modular subgroup lattices.  Recall that a subgroup $M\leq G$ is called \emph{modular} in $G$ if $$ \left\langle A,M\cap B \right\rangle = \left\langle A,M \right\rangle \cap B $$ for all subgroups $A\leq B\leq G$. Thus modularity is a lattice-theoretic weakening of normality. A subgroup $H\leq G$ is called \emph{nearly modular} in $G$ if there exists a modular subgroup $M\leq G$ such that $$ H\leq M \text{ and } [M:H]<\infty. $$ 
We write $\operatorname{Sub}(G)$ for the lattice of all subgroups of $G$, ordered by inclusion, and call it \emph{modular}, respectively \emph{nearly modular}, if every subgroup of $G$ is modular, respectively nearly modular, in $G$.

\subsection{Restricted direct products} \hfill 

\medskip
We first record two elementary facts about direct products. Local solvability is preserved under restricted direct products, but may fail for unrestricted direct products.

 Let $(G_i)_{i \in I}$ be a family of groups. For $j \in I$, the \emph{canonical projection} onto the $j$-th factor is the epimorphism $\pi_j \colon \prod_{i \in I} G_i \twoheadrightarrow G_j$ given by $(x_i)_{i \in I} \mapsto x_j$. The \emph{support} of $x \in \prod_{i \in I} G_i$ is defined by 
$$ \operatorname{supp}(x) = \bigl\{ i \in I \mid \pi_i(x) \neq 1_{G_i} \bigr\}. $$ 
The \emph{restricted direct product} of $(X_i)_{i\in I}$ with $X_i\subseteq G_i$ is
$$ \bigoplus_{i \in I} X_i = \Bigl\{ x \in \prod_{i \in I} X_i \Bigm| \lvert\operatorname{supp}(x)\rvert < \infty \Bigr\}. $$

\begin{lemma}\label{lem:restricteddirectproductlocsolislocsol}
    For any family of locally solvable groups $(H_i)_{i \in I}$, the restricted direct product $\bigoplus_{i \in I} H_i$ is locally solvable.        
\end{lemma}

\begin{proof}
Put $ H = \bigoplus_{i\in I} H_i $, and let $ K = \langle a_1, \dots, a_n \rangle $ be a finitely generated subgroup of $ H $. Observe that each $ a_i $ has finite support. Thus, the index~set $$J = \bigcup_{i=1}^n \operatorname{supp}(a_i)$$ is  finite. Any element $a \in K$ is a word $w(a_1, \dots, a_n)$. Thus, for $i \notin J$,
$$
\pi_i(a) = w\bigl(\pi_i(a_1), \dots, \pi_i(a_n)\bigr) = w(1_{H_i}, \dots, 1_{H_i}) = 1_{H_i}.
$$
Therefore $\operatorname{supp}(a) \subseteq J$, and consequently 
$$K \subseteq \bigl\{h \in H \mid \pi_i(h)=1_{H_i}  \text{ for all } i \in I \setminus J \bigr\}.$$
Observe that the latter subgroup is naturally isomorphic to $ \prod_{j \in J} H_j $ via the map $x \mapsto(\pi_j(x))_{j \in J}.$ Since the product $\prod_{j \in J} H_j$ is locally solvable, the subgroup $K$ is solvable. Therefore $H$ is locally solvable, as required.
\end{proof}

\begin{Remark}
Local solvability need not be preserved under unrestricted direct products.
Indeed, let
$$\mathbf G=\prod_{n\geq 1}\mathbf S_{n,2},$$
where $\mathbf S_{n,2}=F_2/F_2^{(n)}$, with $F_2$ denoting the free group of
rank $2$. Thus $\mathbf S_{n,2}$ is the free solvable group of rank $2$ whose
derived length is exactly $n$.  For each $n\geq 1$, let $x_n$ and $y_n$ be free generators of
$\mathbf S_{n,2}$. Define
$$x=(x_n)_{n\geq 1},
\,
y=(y_n)_{n\geq 1},$$
and put $H=\langle x,y\rangle\leq \mathbf G.$ The restriction to $H$ of the projection onto the $n$th coordinate maps
$x$ and $y$ to $x_n$ and $y_n$, respectively, and is therefore surjective.
Hence $H$ is a subdirect product of the groups $\mathbf S_{n,2}$.
If $H$ is solvable, then every quotient $\mathbf S_{n,2}$ of $H$ is
solvable of derived length at most $\operatorname{dl}(H)$.  ~Hence
$n=\operatorname{dl}(\mathbf S_{n,2})
\leq \operatorname{dl}(H)$
for every $n\geq 1$, which is impossible.  Thus $H$ is a finitely generated
non-solvable subgroup of $\mathbf G$, and consequently $\mathbf G$ is not
locally solvable.
\end{Remark}

\begin{lemma}\label{prop:famlocsolrad}
For an arbitrary family of groups $(H_i)_{i\in I}$, the following statements hold:
    \begin{enumerate}[label = \textup{(\arabic*)}]
         \item  $S\bigl(\bigoplus_{i \in I}H_i\bigr)=\bigoplus_{i \in I}S(H_i).$         \item Suppose that $H_i$ admits the locally solvable radical for every $i \in I$. Then $\bigoplus_{i \in I}H_i$ also admits the locally solvable radical, and
        $$ R_{\textup{L}\mathfrak{S}}\biggl(\bigoplus_{i \in I}H_i\biggr) = \bigoplus_{i\in I} R_{\textup{L}\mathfrak{S}}(H_i). $$
        Moreover, the equality $$R_{\textup{L}\mathfrak{S}} \biggl(\bigoplus_{i \in I}H_i \biggr) = S\biggl(\bigoplus_{i \in I}H_i \biggr)$$ holds if and only if $R_{\textup{L}\mathfrak{S}}(H_i) = S(H_i)$ for every $i \in I$.
    \end{enumerate}
\end{lemma}
    
\begin{proof} Put $H=\bigoplus_{i \in I}H_i$.

(1) Let $ h = (h_i)_{i \in I} \in S(H) $. Fix an index $ i \in I $ and an arbitrary element $ a \in H_i $, and let $ \sigma_i(a) $ denote its canonical embedding in $ G $. Since $ h \in S(H) $, the subgroup $ \langle h, \sigma_i(a) \rangle $ is solvable. Therefore, its image
$$
\pi_i\bigl( \langle h, \sigma_i(a) \rangle \bigr) = \langle h_i, a \rangle
$$
via the projection $\pi_i$
is also solvable. This immediately gives $ h_i \in S(H_i) $. Moreover, since $ h $ has finite support, it follows that $ h \in \bigoplus_{i \in I} S(H_i) $, and thereby the inclusion $ S(G) \subseteq \bigoplus_{i \in I} S(H_i) $ holds.

Now let $h=(h_i)_{i \in I}\in \bigoplus_{i \in I} S(H_i).$ Fix arbitrary $x=(x_i)_{i \in I}\in H$. Since both $h$ and $x$ have finite support, it follows that the set $$J=\operatorname{supp}(h)\cup\operatorname{supp}(x)$$ is finite. Thus, for each $j \in J$, the condition $h_j \in S(H_j)$ implies that the subgroup $\langle h_j,x_j\rangle$ solvable. On the other hand, a word in $h$ and $x$ is computed coordinatewise, and thereby 
$$\langle h,x\rangle \leq \bigoplus_{i \in I}\langle h_i, x_i \rangle.$$ 
Note that $x_i=h_i=1_{H_i}$ for all $i \in I \setminus J$.  Thus, the latter group is naturally isomorphic to $ \prod_{j \in J} \langle h_j, x_j \rangle $. The finiteness of $ J $ guarantees the solvability of this direct product, from which the solvability of $ \langle h,x \rangle $ immediately follows. Consequently $\bigoplus_{i \in I}S(H_i)\subseteq S(H)$.~Therefore
$$S\biggl(\bigoplus_{i \in I}H_i\biggr)=\bigoplus_{i \in I}S(H_i), \eqno{(*)}$$
and this completes part~\textup{(1)}.

\medskip

(2) According to Lemma~\ref{lem:restricteddirectproductlocsolislocsol}, the restricted direct product $ \bigoplus_{i \in I} R_{\LS}(H_i) $ is a normal locally solvable subgroup of $ \bigoplus_{i \in I}H_i $.
Now let $L \trianglelefteq \bigoplus_{i \in I}H_i$ be locally solvable. Since $\pi_i(L) \trianglelefteq R_{\LS}(H_i)$, it follows that
$L \leq \bigoplus_{i\in I} R_{\LS}(H_i)$. Therefore, the locally solvable radical of $H$ exists, and 
$$R_{\LS}\biggl(\bigoplus_{i \in I}H_i\biggr) = \bigoplus_{i\in I} R_{\LS}(H_i). \eqno{(* *)}$$

For the final assertion of part \textup{(2)}, the condition $R_{\textup{L}\mathfrak{S}}(H)=S(H)$, together with the equalities $(*)$ and $(* *)$, implies that
$$\bigoplus_{i\in I} R_{\textup{L}\mathfrak{S}}(H_i) = \bigoplus_{i\in I} S(H_i). $$
Projection onto $H_i$ yields $R_{\textup{L}\mathfrak{S}}(H_i)=S(H_i)$ for all $i\in I$. The converse is immediate, completing the proof of part~\textup{(2)}.
\end{proof}

\begin{lemma} \label{proplocradformula}
    Let $(H_i)_{i\in I}$ be family of groups. Then the following statements~hold:
  \begin{enumerate}[label = \textup{(\arabic*)}]
    \item If each $H_i$ is locally radical, then  $\bigoplus_{i \in I} H_i$ is locally radical.
    \item If each $H_i$ admits the $\textup{L}\mathfrak{R}$-radical, then the restricted direct product $\bigoplus_{i \in I}H_i$ also admits the $\textup{L}\mathfrak{R}$-radical, and
        $$ R_{\textup{L}\mathfrak{R}}\biggl(\bigoplus_{i \in I}H_i\biggr) = \bigoplus_{i\in I} R_{\textup{L}\mathfrak{R}}(H_i). $$
  \end{enumerate}
\end{lemma}

\begin{proof}
Part (1) follows from \cite[Proposition 1.2.15(iii)]{Dixon17}. The initial argument in the proof of Lemma~\ref{prop:famlocsolrad}(2) carries over to part (2) upon replacing locally solvable subgroups with locally radical ones.
\end{proof}

\subsection{The modular case} \hfill

\medskip

We now address the modular case. Here, the lattice condition provides enough structure to compute the relevant radicals~directly.
   
Recall that a \emph{Tarski group} is an infinite group in which every proper non-trivial subgroup is of prime order. An \emph{extended Tarski group} is a group $G$ such that $G/Z(G)$ is a Tarski group of exponent $p$ for some prime $p$, the center $Z(G)$ is cyclic of order $p^r >1$, and for every subgroup $H \leq G$, either $H \leq Z(G)$ or $H \geq Z(G)$ holds. The existence of such groups was proved by Ol'shanskii \cite{Olshanskii80,Olshanskii82,Olshanskii91} for all primes $p>10^{75}$.

\begin{theorem}\label{theo:modularlatticelocalsolrad}
    Let $G$ be a group whose subgroup lattice $\textup{Sub}(G)$ is modular. Then $G$ admits the locally solvable radical, and $$R_{\textup{L}\mathfrak{S}}(G) = R_{\textup{L}\mathfrak{R}}(G) = S(G).$$ 
\end{theorem}

\begin{proof}
We first treat the case where $G$ is not torsion. By \cite[Theorems~5 and~6]{Iwasawa}, the torsion subgroup $T(G)$ and the quotient $G/T(G)$ are both abelian. In particular, $G$ is solvable, and
$$R_{\textup{L}\mathfrak{S}}(G) = R_{\textup{L}\mathfrak{R}}(G) = G = S(G).$$

We may now assume $ G $ to be torsion. In this case, \cite[Theorem]{Schmidt_Gruppenmit} guarantees that $ G $ decomposes as the restricted direct product 
$$G =  \bigoplus_{i \in I} T_i  \times  \bigoplus_{j \in J} E_j  \times L,$$
where the $ T_i $ are Tarski groups, the $ E_j $ are extended Tarski groups, $ L $ is a locally finite group, and elements from distinct direct factors have coprime orders.
Since Tarski groups and extended Tarski groups are Noetherian, their locally
solvable radicals, $\mathrm{L}\mathfrak{R}$-radicals, and solvable radicals
exist and coincide. For locally finite groups, Theorem~\ref{theo:locally-finite} establishes that the locally solvable radical coincides with the $\mathrm{L}\mathfrak{R}$-radical. Combining these observations with Lemmas~\ref{prop:famlocsolrad}
and~\ref{proplocradformula}, we conclude that the corresponding radicals
of $G$ exist and are given by $$R_{\textup{L}\mathfrak{S}}(G) = R_{\textup{L}\mathfrak{R}}(G) = \bigoplus_{i \in I} R(T_i) \times \bigoplus_{j \in J} R(E_j) \times R_{\textup{L}\mathfrak{S}}(L).$$ 
Moreover, \cite[Theorem~3.4]{CNH2025} gives
$R_{\mathrm{L}\mathfrak{S}}(T_i)=S(T_i)$ and
$R_{\mathrm{L}\mathfrak{S}}(E_j)=S(E_j)$ for all $i\in I$ and $j\in J$, respectively, while Theorem~\ref{theo:locally-finite} gives $R_{\mathrm{L}\mathfrak{S}}(L)=S(L).$
  An additional application of Lemma~\ref{prop:famlocsolrad} then implies that
\begin{align*}
R_{\textup{L}\mathfrak{S}}(G) =R_{\textup{L}\mathfrak{R}}(G)
&= \bigoplus_{i\in I} S(T_i) \times \bigoplus_{j\in J} S(E_j) \times S(L) \\
&= S(G),
\end{align*}
which completes the proof.
\end{proof}

\subsection{Lifting from the modular case} \hfill

\medskip
Passing from modular to nearly modular subgroup lattices requires a lifting argument for solvable-by-finite normal subgroups. Every solvable-by-finite group admits the solvable radical by \cite[Theorem~2.8]{CNH2025}. The main point is to control elements lying above the radical of the quotient through their conjugation action on the solvable radical of the kernel. The following theorem provides the required lifting statement.

\begin{theorem}\label{prop:main_tool}
Let $N$ be a solvable-by-finite normal subgroup of $G$, and let
  $\pi\colon G\twoheadrightarrow \overline{G}:=G/N$ be the canonical epimorphism. Assume that $R_{\mathrm{L}\mathfrak{S}}(\overline{G})$ exists and that $R_{\mathrm{L}\mathfrak{S}}(\overline{G})=S(\overline{G}).$ Set $$C=\left\{\,g\in G\mid [g,N]\leq R(N)\,\right\}.$$ Then the following assertions hold:
\begin{enumerate}[label = \textup{(\arabic*)}]
    \item $R(N) = C \cap N$.
    \item The group $G$ admits the locally solvable radical, and 
    $$ R_{\textup{L}\mathfrak{S}}(G) = C \cap \pi^{-1}(R_{\textup{L}\mathfrak{S}}(\overline{G}))=S(G). $$
    \item If, in addition, $\overline{G}$ admits the $\textup{L}\mathfrak{R}$-radical and $$ R_{\textup{L}\mathfrak{R}}(\overline{G})=R_{\textup{L}\mathfrak{S}}(\overline{G}), $$ then $G$ admits the $\textup{L}\mathfrak{R}$-radical and $ R_{\textup{L}\mathfrak{R}}(G) = R_{\textup{L}\mathfrak{S}}(G). $
\end{enumerate}
\end{theorem}
    
\begin{proof}
(1)
The condition $[g,N]\leq R(N)$ is invariant under conjugation by elements
of $G$, and hence $C\trianglelefteq G$. Since $R(N)$ is characteristic in
$N$ and $N\trianglelefteq G$, we have $R(N)\trianglelefteq G$, and hence $R(N)\leq C\cap N.$ Conversely, let $n\in C\cap N$.~Then $[n,N]\leq R(N)$, so $nR(N)\in Z\bigl(N/R(N)\bigr).$  Moreover, $N/R(N)$ is semisimple, and hence
$Z\bigl(N/R(N)\bigr)=1.$ Therefore $n\in R(N)$, and so $C\cap N=R(N).$

\medskip
(2) Put  $J=C\cap \pi^{-1}(R_{\textup{L}\mathfrak{S}}(\overline{G}))$. Let $F$ be a finitely generated subgroup of $J$. Then $\pi(F)\leq R_{\textup{L}\mathfrak{S}}(\overline{G}),$ which implies that $\pi(F)$ is solvable. Moreover, part (1) ensures that the kernel of the restriction $\pi_{|_F}$ satisfies
$$\ker\pi_{|_F} = F \cap N \leq C \cap N = R(N).$$ 
Consequently, $F \cap N$ is solvable, and hence $F$ itself is solvable. In particular, $J$ is locally solvable. Now, let $L$ be an normal locally solvable subgroup of~$G$. According to \cite[Lemma~2.5]{CNH2025}, the subgroup $L \cap N$ is solvable, and thus $L \cap N \leq R(N) $. Since $L,N\trianglelefteq G$, we have $[L,N]\leq L\cap N\leq R(N),$ and hence $L\leq C$. Moreover, $\pi(L)\leq R_{\mathrm{L}\mathfrak S}(\overline G).$ By the definition of $J$, it follows that $L\leq J$. Therefore $J$ is the largest normal locally solvable subgroup of $G,$ and hence $R_{\textup{L}\mathfrak{S}}(G) = J.$

\medskip

To complete part~\textup{(2)}, it remains to prove that $S(G)=J$. We proceed in two steps.

\smallskip
\emph{Step 1.} Let $g \in S(G)$. Then $g \in \pi^{-1}(R_{\textup{L}\mathfrak{S}}(\overline{G})).$ Passing to the quotients $ \widetilde{G} = G/R(N)$ and $\widetilde{N} = N/R(N),$ and writing $ \widetilde{g} = gR(N)$, we obtain that $\widetilde{N}$ is a finite group with $R(\widetilde{N}) = \{1\}.$
Let $n \in N$, and let $\widetilde{n}$ be its image in $\widetilde{N}$. Since $g \in S(G)$, it follows that $\langle g, n \rangle$ is solvable, hence so is its image $\langle \widetilde{g}, \widetilde{n} \rangle$.

Now consider the subgroup $X=\langle \widetilde{N},\widetilde{g}\rangle.$ The finiteness of $\widetilde{N}$ implies that the conjugation action of $\widetilde{g}$ on $ \widetilde{N}$ has finite order, so there exists an integer $m \geq 1$ such that $\widetilde{g}^{\,m} $ centralizes $\widetilde{N}$. Put $Y = X / \langle \widetilde{g}^{\,m} \rangle$, and let $ \overline{g}$ denote the image of $ \widetilde{g}$ in $Y$. Since $\widetilde{g}$ normalizes $ \widetilde{N}$, every element of $X$ can be written in the form $\widetilde{n}\widetilde{g}^{\,k}$ for some $ \widetilde{n} \in \widetilde{N}$ and some $k \in \mathbb{Z}$. Let $\overline{N}$ denote the image of $\widetilde{N}$ in $Y$. The relation $\overline{g}^{\,m} = 1$ forces the order of $\overline{g}$ to divide $m$. Consequently $ \overline{g}$ generates a cyclic subgroup of order at most $m$, and hence $Y$ is finite. Note that the subgroup $\langle\widetilde{g},\widetilde{n}\rangle$ is solvable, and hence so is its image $\langle\overline{g},\overline{n}\rangle$. Moreover, the equality $\langle \overline{g}, \overline{n}\,\overline{g}^{\,k} \rangle = \langle \overline{g}, \overline{n} \rangle$ holds for all $k \in \mathbb{Z}$, thereby $\overline{g}\in S(Y)$. In view of \cite[Theorem~1.1]{Guralnick06},  $S(Y) = R(Y)$, and hence $\overline{g} \in R(Y)$. 

On the other hand, since $\widetilde{g}^{\,m}$ centralizes $\widetilde{N}$, the cyclic subgroup $ \langle \widetilde{g}^{\,m} \rangle $ is central in $X$. As a result, the kernel of the natural projection $\widetilde{N} \to \overline{N}$, namely $ \widetilde{N} \cap \langle \widetilde{g}^{\,m} \rangle $, is contained in $Z(\widetilde{N})$. The triviality of $ R(\widetilde{N})$ implies that $Z(\widetilde{N}) = 1$, and so $\widetilde{N} \cap \langle \widetilde{g}^{\,m} \rangle=1$. We thus obtain an isomorphism $\overline{N} \cong \widetilde{N}$, and hence $ R(\overline{N}) = 1$. This immediately yields
$$[R(Y), \overline{N}] \leq R(Y) \cap \overline{N} = 1.$$
Consequently $\overline{g}$ centralizes $\overline{N}$, forcing $[\widetilde{g}, \widetilde{N}] \leq \langle \widetilde{g}^{\,m} \rangle \cap \widetilde{N} = 1$. Lifting this relation back to $G$ yields $[g,N] \leq R(N)$, which confirms $S(G) \subseteq J$.

\smallskip
\emph{Step 2.}
Let $g \in J $ and fix an arbitrary element $x \in G$. 
Put $H=\langle g,x\rangle$.
The normality of $C$ in $G$ implies that $H \cap C \trianglelefteq H$ and the quotient $H/(H \cap C)$ is generated by the image of $x$. 
Since $ g \in J$, it follows that $\pi(g) \in  S(\overline{G})$.
Consequently $\pi(H) = \langle \pi(g), \pi(x) \rangle$ is solvable, and thereby its subgroup $\pi(H \cap C)$ is also solvable.
On the other hand,  part~\textup{(1)} implies that $$\ker\pi_{|_{H \cap C}} = H \cap C \cap N \leq R(N).$$
Thus $H \cap C \cap N$ is solvable. 
Since both $H\cap C\cap N$ and $$\pi(H\cap C) \cong \frac{H\cap C}{H\cap C\cap N}$$
are solvable, $H\cap C$ is solvable.  Hence $H$ is solvable, and so $J\subseteq S(G).$

\medskip
(3) Suppose that $\overline G$ admits the $\textup{L}\mathfrak{R}$-radical satisfying $R_{\textup{L}\mathfrak R}(\overline G)=R_{\textup{L}\mathfrak S}(\overline G)$, and put $K=C\cap \pi^{-1}\bigl(R_{\textup{L}\mathfrak R}(\overline G)\bigr)$. 
Together with part~\textup{(2)}, the equalities
$$R_{\textup{L}\mathfrak R}(\overline G) = S(\overline G) = R_{\textup{L}\mathfrak S}(\overline G)$$ yield $K = C\cap \pi^{-1}\bigl(R_{\textup{L}\mathfrak S}(\overline G)\bigr) = R_{\textup{L}\mathfrak S}(G).$ 
Thus $K$ is a normal locally solvable subgroup of $G$, and hence a normal locally radical subgroup. 
Therefore $$K\leq R_{\textup{L}\mathfrak R}(G).$$

Conversely, suppose $L$ is a normal locally radical subgroup of $G$. 
Since $L\cap N$ is a normal locally radical subgroup of $N$, we have
$$L\cap N\leq R_{\mathrm{L}\mathfrak{R}}(N).$$
Moreover, as $N$ is solvable-by-finite, \cite[Lemma~2.5 and Theorem~2.6]{CNH2025} gives
$$R_{\mathrm{L}\mathfrak{R}}(N)\leq R(N).$$
Therefore, $L\cap N\leq R(N).$ For $l\in L$ and $n\in N$, we have
$$[l,n]\in L\cap N\leq R(N),$$
and hence $L\leq C$. 
Moreover, $\pi(L)$ is a normal locally radical subgroup of $\overline{G}$, so $\pi(L)\leq R_{\mathrm{L}\mathfrak{R}}(\overline{G}).$ 
By the definition of $K$, we conclude that $L\leq K$.
Since $K$ contains every normal locally radical subgroup of $G$ and is itself such a subgroup, $K$ is the largest normal locally radical subgroup of $G$, which yields $K=R_{\textup{L}\mathfrak R}(G)$. 
Combining this with the equality $K=R_{\textup{L}\mathfrak S}(G),$ we immediately obtain
$R_{\textup{L}\mathfrak R}(G) = R_{\textup{L}\mathfrak S}(G),$ as desired.
\end{proof}

\subsection{The nearly modular case} \hfill

\medskip
We now specialize the preceding lifting criterion to groups with nearly modular subgroup lattices. We begin by stating the preservation properties of the lattice condition.

\begin{lemma}\label{lem:modularpreserve}
Let $G$ be a group with nearly modular subgroup lattice. Then the following assertions hold:
\begin{enumerate}[label=\textup{(\arabic*)}]
    \item If $S$ is a subgroup of $G$, then  $\operatorname{Sub}(S)$ is nearly modular.
    \item If $N$ is a normal subgroup of $G$, then $\operatorname{Sub}(G/N)$ is nearly modular.
\end{enumerate}
\end{lemma}
    
\begin{proof}
(1) Let $K \leq S \leq G$. Since $\textup{Sub}(G)$ is nearly modular, the subgroup $K$ is nearly modular in $G$. Let $M$ be a modular subgroup of $G$ containing $K$ with $[M:K] < \infty$. Then $[M \cap S:K] \leq [M:K] < \infty.$ It remains to verify that $M \cap S$ is modular in $S$. For any subgroups $A \leq B \leq S$, the following inequalities hold: $$\langle A, M \cap B \rangle \leq \langle A, M \cap S \rangle \cap B \leq \langle A, M \rangle \cap B.$$ By the modularity of $M$ in $G$,  we have $\langle A, M \rangle \cap B = \langle A, M \cap B \rangle.$ Thus, $$\langle A, M \cap S \rangle \cap B  = \langle A, M \cap B \rangle = \langle A, M \cap S \cap B \rangle,$$ so the subgroup $M \cap S$ is modular in $S$. Therefore $\textup{Sub}(S)$ is nearly modular.

\medskip

(2)  Let $\overline K\leq G/N$. Then
$\overline K=K/N$ for some subgroup $K\leq G$ with $N\leq K$.
Since $\operatorname{Sub}(G)$ is nearly modular, there exists a modular subgroup
$M\leq G$ such that $K\leq M$ and $[M:K]<\infty.$
In view of \cite[Proposition 7.1.5]{Humphreys}, the natural
lattice isomorphism
    $$[N,G]\rightarrow \operatorname{Sub}(G/N), \, X\mapsto X/N,$$ preserves modular elements. 
    Hence $M/N$ is modular in $G/N$. Moreover 
    $$K/N\leq M/N \text{ and } [M/N:K/N]=[M:K]<\infty.$$
Thus $K/N$ is nearly modular in $G/N$, making $\textup{Sub}(G/N)$ nearly modular.
\end{proof}

We next invoke Neumann's finiteness criterion, which converts a lattice-theoretic condition into a group-theoretic constraint on the derived subgroup. Recall that a subgroup $H \leq G$ is \emph{nearly normal} if $[H^G:H] < \infty$.

\begin{proposition}[{\cite[Theorem 13.1]{Neumann_groupswithfiniteclass}}]\label{theo:everysubgroupisnearlynormal}
    Let $G$ be a group. Every subgroup of $G$ is nearly normal if and only if the derived subgroup $[G,G]$ is finite. 
\end{proposition}

\begin{lemma}\label{lem:nonpernearmodularlatfinabel}
Let $G$ be a non-torsion group with nearly modular subgroup lattice. Then the following assertions hold:
\begin{enumerate}[label = \textup{(\arabic*)}]
    \item The set of torsion elements $T(G)$ is a characteristic finite-by-abelian subgroup of $G$.
    \item The quotient group $G/T(G)$ admits the locally solvable radical, and 
    $$ R_{\textup{L}\mathfrak{S}}\bigl(G/T(G)\bigr) = R_{\textup{L}\mathfrak{R}}\bigl(G/T(G)\bigr) = S\bigl(G/T(G)\bigr) = Z\bigl(G/T(G)\bigr). $$
\end{enumerate}
\end{lemma}

\begin{proof} (1)
Since $\textup{Sub}(G)$ is nearly modular, every cyclic subgroup of $G$ is nearly modular. According to \cite[Proposition~1]{Giovanni}, the set $T(G)$ of torsion elements in $G$ is a characteristic subgroup of $G$. 

Let $H$ be a subgroup of $T(G)$. Since $H$ is torsion and $G$ is non-torsion, \cite[Corollary~9]{Giovanni} now implies that $H$ is nearly normal in $ G $, i.e.
$[H^G : H] < \infty.$ Moreover, as $ H \leq H^{T(G)} \leq H^G $, it immediately follows from \cite[1.7.11]{Scott2010} that $$[H^{T(G)} : H] \leq [H^G : H] < \infty.$$  Consequently, every subgroup of $ T(G) $ is nearly normal in $ T(G) $. In view of Proposition \ref{theo:everysubgroupisnearlynormal}, the derived subgroup $ [T(G),T(G)] $ is finite, or equivalently, that $ T(G) $ is finite-by-abelian. 

\medskip
(2) Put $Q = G/T(G)$. We claim that $S(Q)=Z(Q)$. Indeed, let $a\in S(Q)$ and $x\in Q \setminus \{1\}$, and put $H=\langle a,x\rangle$. Then $H$ is a finitely generated solvable subgroup of $Q$. Moreover $H$ is torsion-free, and its subgroup lattice is nearly modular by Lemma~\ref{lem:modularpreserve}. According to \cite[Theorem~10]{Giovanni}, $H$ is abelian. In particular, $ax=xa$, from which $a\in Z(Q)$.~Therefore $S(Q)\subseteq Z(Q)$.

It remains to determine $R_{\mathrm{L}\mathfrak R}(Q)$. Let $L$ be a
nontrivial normal locally radical subgroup of $Q$. By Lemma~\ref{lem:modularpreserve}, the lattice $\operatorname{Sub}(L)$ is nearly modular. Moreover, every locally radical group is locally graded (see, for example, \cite[comments preceding Corollary~4.3.9]{Dixon17}).
Therefore, \cite[Theorem~10]{Giovanni} implies that $L$ is abelian.
Now let $b \in L$ and $y \in Q$, and put $K = \langle b, y \rangle$. The normal closure $N = \langle b^{\langle y \rangle} \rangle$ lies in $L$. Since $N \le L$ and $K/N$ is cyclic, it follows that $K$ is solvable, and thereby $L \le Z(Q)$. Moreover, since every locally solvable group is locally radical, we have
$R_{\mathrm{L}\mathfrak S}(Q)
\leq
R_{\mathrm{L}\mathfrak R}(Q).$
Conversely, the preceding argument shows that every normal locally radical
subgroup of $Q$ is abelian, and hence locally solvable. Therefore,
$$R_{\mathrm{L}\mathfrak R}(Q) \leq R_{\mathrm{L}\mathfrak S}(Q).$$
Combining these inclusions, we obtain
$R_{\mathrm{L}\mathfrak S}(Q) =R_{\mathrm{L}\mathfrak R}(Q) = Z(Q).$
\end{proof}

We now state the main result of this section.

\begin{theorem}\label{thm:target_3}
    Let $G$ be a group whose lattice of subgroups $\textup{Sub}(G)$ is nearly modular. Then $G$ admits the locally solvable radical, and 
    $$ R_{\textup{L}\mathfrak{S}}(G) = R_{\textup{L}\mathfrak{R}}(G) = S(G). $$
\end{theorem}

\begin{proof}
We claim that $G$ has a normal solvable-by-finite  subgroup $N$ such~that
$$R_{\mathrm{L}\mathfrak{S}}(G/N) = R_{\mathrm{L}\mathfrak{R}}(G/N) = S(G/N).$$
The desired conclusion then follows from Theorem~\ref{prop:main_tool}. If $G$ is non-torsion, Lemma~\ref{lem:nonpernearmodularlatfinabel} implies that the torsion subgroup $T(G)$ is characteristic and finite-by-abelian, and that $G/T(G)$ satisfies
$$R_{\mathrm{L}\mathfrak{S}}(G/T(G)) = R_{\mathrm{L}\mathfrak{R}}(G/T(G)) = S(G/T(G)).$$
Since finite-by-abelian groups are solvable-by-finite, $N = T(G)$ fulfills the claim. If $G$ is torsion, a theorem of Falco \cite{Falco2003} yields a finite normal subgroup $N \trianglelefteq G$ such that the lattice $\mathrm{Sub}(G/N)$ is modular. By Theorem~\ref{theo:modularlatticelocalsolrad}, $G/N$ satisfies the required equalities, which proves the claim.

In either case, such a subgroup $N$ exists. The proof is now complete.
\end{proof}

\begin{corollary}
Let $G$ be a group whose subgroup lattice $\textup{Sub}(G)$ is nearly modular. Then $G$ is $R_{\textup{L}\mathfrak{S}}$-hereditary. Moreover, if $A$ is a locally radical ascendant subgroup of $G$, then its normal closure $\langle A^G \rangle$ is locally solvable.
\end{corollary}

\begin{proof}
By Lemma~\ref{lem:modularpreserve}, the subgroup lattice of every subgroup
of $G$ is nearly modular. Hence Theorem~\ref{thm:target_3} shows that $G$ is
$R_{\mathrm{L}\mathfrak{S}}$-hereditary. Lemma~\ref{lem:ascenlocsolsubgroupislocsol} then implies that the normal closure
$\langle A^G\rangle$ is locally solvable.
\end{proof}

We conclude by asking whether Wilson radical sets depend on the choice of
$\omega\in\Omega$. Outside the classes covered by Theorems~\ref{theo:locally-finite} and \ref{theo:locradislocsol}, it remains
unknown whether $\mathcal W_\omega(G)$ is independent of $\omega$.
Let $T$ be a Tarski monster and let $\omega\in\Omega$. Since $T$ is nonabelian simple, it follows from Lemma~\ref{lem:SsubseteqW} that 
$$\langle\mathcal W_\omega(T)\rangle\in\{1,T\}.$$
The first case is equivalent to $\mathcal W_\omega(T)=\{1\}$, whereas the second only asserts that $\mathcal W_\omega(T)$ generates $T$. This motivates the following problem.

\begin{problem}
\label{quest:tarski_wilson}
Do there exist $\omega_+,\omega_-\in\Omega$ such that, for every Tarski
monster $T$,
  $$\mathcal W_{\omega_+}(T)=T \text{ and } \mathcal W_{\omega_-}(T)=\{1\}?$$
\end{problem}

\end{document}